\documentclass{article}

\usepackage{amsmath,amscd,amsfonts,amssymb,amsthm}
\usepackage{graphicx}
\usepackage{epstopdf}
\usepackage{a4wide}
\usepackage{graphicx}
\usepackage{subfigure}
\usepackage[latin1]{inputenc}
\usepackage{float}
\usepackage{placeins}
\usepackage{multirow}

\newcommand{\subj}[1]{\par\noindent{\bf Mathematics Subject Classification 2010: }#1.}
\newcommand{\keyw}[1]{\par\noindent{\bf Keywords: }#1.}

\theoremstyle{definition}
\newtheorem{definition}{Definition}
\newtheorem{theorem}{Theorem}
\newtheorem{example}{Example}
\theoremstyle{remark}

\def\a{\alpha}
\def\t{\triangle T}
\def\DS{\displaystyle}
\def\LHD{{_a\mathcal{D}_t^\a}}
\def\HD{{_1\mathcal{D}_t^{\a}}}
\def\RHD{{_t\mathcal{D}_b^\a}}

\begin{document}

\title{A Discretization of the Hadamard fractional derivative}

\author{Ricardo Almeida$^1$\\
{\tt ricardo.almeida@ua.pt}
\and
 Nuno R. O. Bastos$^{1,2}$\\
{\tt nbastos@estv.ipv.pt}}

\date{$^1$Center for Research and Development in Mathematics and Applications (CIDMA)\\
Department of Mathematics, University of Aveiro, 3810--193 Aveiro, Portugal\\
$^2$Department of Mathematics, School of Technology and Management of Viseu \\
Polytechnic Institute of Viseu, 3504--510 Viseu, Portugal}

\maketitle


\begin{abstract}
We present a new discretization for the Hadamard fractional derivative, that simplifies the computations. We then apply the method
 to solve a fractional differential equation  and a fractional variational problem with dependence on the Hadamard fractional derivative.
\end{abstract}

\subj{26A33, 49M25, 49M25}

\keyw{fractional calculus, discretization methods}


\section{Introduction}

Similarly to ordinary calculus, we can find in the literature distinct definitions for fractional derivatives and for fractional integrals, which are
 generalizations of the integer-order derivatives and multiple integrals, respectively. The most common ones and consequently more studied are the Riemann--Liouville,
 Caputo and Gr\"{u}nwald--Letnikov  definitions. We deal in this paper with the Hadamard fractional derivative, introduced in \cite{Hadamard}. Recently, it has
 call the attention of researchers and numerous results have appeared, with an extensive study of properties of such kind of operators
\cite{Butzer,Butzer1,Kilbas3}.
 For recent results we suggest \cite{Jarad,Qassim,Qian}.

 Due the complexity of solving equations involving fractional operators, in most cases is impossible to determine the exact solution and so numerical methods
 are used to determine an approximated solution of the problem.
 This is an emerging field, and we can find already in the literature several methods to deal with these problems,
 at least for the most common fractional derivative types.
 For the Hadamard fractional derivative, we mention the recent paper \cite{Almeida}, where the fractional operator is replaced by a finite sum
 involving only integer-order derivatives of the function. Replacing the fractional derivative by this sum, we rewrite the initial problem in terms of
 integer-order derivatives and thus we are able to apply classical known methods. In \cite{Butzer2} another approximation formula is obtained, using also
 integer-order derivatives only. The disadvantage is that in order to have a good approximation, we need to use higher-order derivatives, which
 may not be adequate for fractional problems. In this paper we follow a different path, by discretizing the fractional derivative, and then
  convert continuous problems into discrete ones.

  To start, let us recall the definition of the Hadamard fractional derivative.

\begin{definition}
Let $a,b$ be two reals with  $0<a<b$, and $x:[a,b]\to\mathbb{R}$ be a function.
For $\a\in(0,1)$, the left Hadamard fractional derivative of order $\a$  is defined by
$$\LHD x(t)=\frac{t}{\Gamma(1-\alpha)}\frac{d}{dt}\int_a^t \left(\ln\frac{t}{\tau}\right)^{-\alpha}\frac{x(\tau)}{\tau}\,d\tau,$$
while the right Hadamard fractional derivative of order $\a$  by
$$\RHD x(t)=\frac{-t}{\Gamma(1-\alpha)}\frac{d}{dt}\int_t^b \left(\ln\frac{\tau}{t}\right)^{-\a}\frac{x(\tau)}{\tau}d\tau,$$
where $\Gamma$ denotes the Gamma function.
\end{definition}

When $x$ is an absolutely continuous function, there exists an equivalent definition (cf. \cite{Kilbas2})
$$\LHD x(t)=\frac{x(a)}{\Gamma(1-\a)} \left(\ln\frac{t}{a}\right)^{-\alpha}
+\frac{1}{\Gamma(1-\alpha)}\int_a^t \left(\ln\frac{t}{\tau}\right)^{-\alpha}\dot{x}(\tau)\,d\tau,$$
and
$$\RHD x(t)=\frac{x(b)}{\Gamma(1-\a)}\left(\ln\frac{b}{t}\right)^{-\alpha}
-\frac{1}{\Gamma(1-\alpha)}\int_t^b \left(\ln\frac{\tau}{t}\right)^{-\a}\dot{x}(\tau)d\tau.$$

More properties can be found in references at the end.
The paper is organized in the following way. In Section~\ref{sec:2} we present the main result of the paper. Starting with the definition,
and with an appropriate
grid on time, we present a new discrete version for the left Hadamard fractional derivative. To show the efficiency of the method,
in Section \ref{sec:example} we
compare the exact expression of a fractional derivative with some numerical experiments, for different values of $\a$ and different step sizes $n$.
In Section
\ref{sec:app} we appply the technique to solve a fractional differential equation and a fractional calculus of variation problem.

\section{The discretization method}
\label{sec:2}

The discretization method is described in the following way. Given a function $x:[a,b]\to\mathbb{R}$ , fix a positive integer $n$, and define the time step
$$\triangle T=\frac{\ln\frac{b}{a}}{n}.$$
Given $N\in\{0,1,\ldots,n\}$, denote the time and space grid by
$$t_N=a \exp(N\t)=a\sqrt[n]{\left(\frac{b}{a}\right)^N} \quad \mbox{and} \quad x_N=x(t_N).$$

\begin{theorem} Let $x:[a,b]\to\mathbb{R}$ be a function of class $C^2$ and $n\in\mathbb N$. Denote
$$\psi=\frac{(\t)^{1-\a}}{a(1-\exp(-\t))\Gamma(2-\alpha)} \quad \mbox{and} \quad \left(\omega_k^\a\right)=k^{1-\a}-(k-1)^{1-\a}.$$
Then, for all $N\in\{1,\ldots,n\}$,
$${_a\mathcal{D}_{t_N}^\a}x_N=\tilde{{_a\mathcal{D}_{t_N}^\a}} x_N+  O(\t)  ,$$
where
\begin{equation}\label{appr}
\tilde{{_a\mathcal{D}_{t_N}^\a}} x_N=\frac{x(a)}{\Gamma(1-\a)} \left(\ln\frac{t_N}{a}\right)^{-\alpha}
+\psi\sum_{k=1}^N\left(\omega_{N-k+1}^\a\right)\frac{x_k-x_{k-1}}{\exp(k\t)}\cdot t_k,
\end{equation}
and
$$\lim_{\t\to0}O(\t)=0.$$
\end{theorem}

\begin{proof}
$$\begin{array}{ll}
{_a\mathcal{D}_{t_N}^\a}x_N& = \DS \frac{x(a)}{\Gamma(1-\a)} \left(\ln\frac{t_N}{a}\right)^{-\alpha}
+\frac{1}{\Gamma(1-\alpha)}\int_a^{t_N} \left(\ln\frac{t_N}{\tau}\right)^{-\alpha}\dot{x}(\tau)\,d\tau\\
& = \DS \frac{x(a)}{\Gamma(1-\a)} \left(\ln\frac{t_N}{a}\right)^{-\alpha}
+\frac{1}{\Gamma(1-\alpha)}\sum_{k=1}^N\int_{t_{k-1}}^{t_k} \left(\ln\frac{t_N}{\tau}\right)^{-\alpha}\frac{1}{\tau}\left(\dot{x}(\tau)\tau\right)\,d\tau\\
& \approx \DS \frac{x(a)}{\Gamma(1-\a)} \left(\ln\frac{t_N}{a}\right)^{-\alpha}\\
&\quad \DS +\frac{1}{\Gamma(1-\alpha)}\sum_{k=1}^N\int_{t_{k-1}}^{t_k} \left(\ln\frac{t_N}{\tau}\right)^{-\alpha}\frac{1}{\tau}
\left(\frac{x_k-x_{k-1}}{a(1-\exp(-\t))\exp(k\t)}\cdot t_k\right)\,d\tau\\
& = \DS \frac{x(a)}{\Gamma(1-\a)} \left(\ln\frac{t_N}{a}\right)^{-\alpha}\\
&\quad \DS +\frac{1}{a(1-\exp(-\t))\Gamma(1-\alpha)}\sum_{k=1}^N\frac{x_k-x_{k-1}}{\exp(k\t)}\cdot t_k
\int_{t_{k-1}}^{t_k} \left(\ln\frac{t_N}{\tau}\right)^{-\alpha}\frac{1}{\tau}\,d\tau\\
& = \DS \frac{x(a)}{\Gamma(1-\a)} \left(\ln\frac{t_N}{a}\right)^{-\alpha}\\
&\quad \DS +\frac{(\t)^{1-\a}}{a(1-\exp(-\t))\Gamma(2-\alpha)}\sum_{k=1}^N\frac{x_k-x_{k-1}}{\exp(k\t)}\cdot t_k
\left[(N-k+1)^{1-\a}-(N-k)^{1-\a}\right].\\
\end{array}$$
Thus, we get the desired approximation formula. Now, let us determine an upper bound for the error when we use formula \eqref{appr}. The error is given by
$$E=\frac{1}{\Gamma(1-\alpha)}\sum_{k=1}^N\int_{t_{k-1}}^{t_k} \left(\ln\frac{t_N}{\tau}\right)^{-\alpha}\frac{1}{\tau}
\left|\dot{x}(\tau)\tau-\frac{x_k-x_{k-1}}{t_k-t_{k-1}}\cdot t_k\right|\,d\tau.$$
Let
$$M_i=\max_{\tau\in[a,b]}\left|x^{(i)}(\tau)\right|\, , \quad i=1,2.$$
Then, using Taylor's Theorem, we get that, for all $k\in\{1,\ldots,N\}$ and for all $\tau\in[t_{k-1},t_k]$,
$$\begin{array}{l}
\DS\left|\dot{x}(\tau)\tau-\frac{x_k-x_{k-1}}{t_k-t_{k-1}}\cdot t_k\right|\\
\DS=\left|\dot{x}(\tau)\tau-\left(\dot{x}(t_{k-1})+\ddot{x}(\xi_1)\frac{t_k-t_{k-1}}{2}\right)\cdot t_k\right|\\
\DS\leq \left|\dot{x}(\tau)\tau-\dot{x}(t_{k-1})t_k\right|+M_2\frac{t_k-t_{k-1}}{2}\cdot t_k\\
\DS=\left|\left(\dot{x}(t_{k-1})+\ddot{x}(\xi_2)(\tau-t_{k-1})\right)\tau-\dot{x}(t_{k-1})t_k\right|+M_2\frac{t_k-t_{k-1}}{2}\cdot t_k\\
\DS\leq M_1(t_k-\tau)+M_2(\tau-t_{k-1})\tau+M_2\frac{t_k-t_{k-1}}{2}\cdot t_k\\
\DS\leq (t_k-t_{k-1})\left[M_1+\frac32 M_2b\right].
\end{array}$$
Therefore, the error is bounded by
$$\begin{array}{ll}E& \leq\displaystyle \frac{1}{\Gamma(1-\alpha)}\sum_{k=1}^N\int_{t_{k-1}}^{t_k} \left(\ln\frac{t_N}{\tau}\right)^{-\alpha}\frac{1}{\tau}
(t_k-t_{k-1})\left[M_1+\frac32 M_2b\right] \,d\tau\\
&=\displaystyle  \frac{M_1+\frac32 M_2b}{\Gamma(1-\alpha)}\sum_{k=1}^N  (t_k-t_{k-1})
     \int_{t_{k-1}}^{t_k} \left(\ln\frac{t_N}{\tau}\right)^{-\alpha}\frac{1}{\tau}\, d\tau.
\end{array}$$
Having into consideration that, for all $k\in\{1,\ldots,N\}$ and for all $\tau\in[t_{k-1},t_k)$,
$$0\leq \left(\ln\frac{t_N}{\tau}\right)^{-\alpha} \leq \left(\ln\frac{t_k}{\tau}\right)^{-\alpha},$$
we have that
$$0 \leq  \int_{t_{k-1}}^{t_k} \left(\ln\frac{t_N}{\tau}\right)^{-\alpha}\frac{1}{\tau}\, d\tau\leq \frac{(\t)^{1-\a}}{1-\a}.$$
Also, since
$$t_k-t_{k-1}=a(1-\exp(-\t))\exp(k\t),$$
then
$$\begin{array}{ll}E& \leq\displaystyle
\frac{M_1+\frac32 M_2b}{\Gamma(2-\alpha)}a(1-\exp(-\t))(\t)^{1-\a} \sum_{k=1}^N  \exp(k\t)\\
&=\displaystyle \frac{M_1+\frac32 M_2b}{\Gamma(2-\alpha)}a(\t)^{1-\a}(\exp(N\t)-1)\\
&\leq\displaystyle \frac{M_1+\frac32 M_2b}{\Gamma(2-\alpha)}a(\t)^{1-\a}(\frac{b}{a}-1).\\
     \end{array}$$
In conclusion, we obtain the upper bound formula for our approximation \eqref{appr}:
\begin{equation}
\label{Error:Max}
E \leq \frac{M_1+\frac32 M_2b}{\Gamma(2-\alpha)}(b-a)(\t)^{1-\a}
\end{equation}
which converges to zero as $\t\to0$.
\end{proof}

In opposite to the classical case, where the concept of derivative is local, a fractional derivative contains memory, and thus to compute the
approximation obtained in Eq. \eqref{appr} at a point $t_N$, we need to know the values of $x(t_n)$ from the beginning until de end-point, i.e., from $n=0$
 to $n=N.$

For the right Hadamard fractional derivative, we have in a similar way the following approximation formula:
$${_{t_N}\mathcal{D}_b^\a} x_N\approx\frac{x(b)}{\Gamma(1-\a)} \left(\ln\frac{b}{t_N}\right)^{-\alpha}
-\psi\sum_{k=N+1}^{n}\left(\omega_{k-N}^\a\right)\frac{x_k-x_{k-1}}{\exp(k\t)}\cdot t_k.$$

\section{Example}\label{sec:example}

Let $x(t)=\ln t$, for $t\in[1,2]$. Then (see \cite{Kilbas})
$$\HD x(t)=\frac{(\ln t)^{1-\a}}{\Gamma(2-\a)}.$$
In Figure \ref{IntExp} we show the accuracy of the procedure, for different values of $\a\in\{0.2,0.5,0.7,0.9\}$ and for
different values of $n\in\{10,30,50\}$. The error of the numerical experiments is measured using the norm

\begin{equation}\label{Error}
d(x,y)=\frac{\sum_{k=1}^n|x_k-y_k|}{n}.
\end{equation}

\begin{figure}[h!]
  \begin{center}
    \subfigure[$\a=0.2$]{\label{fig1_c}\includegraphics[scale=0.5]{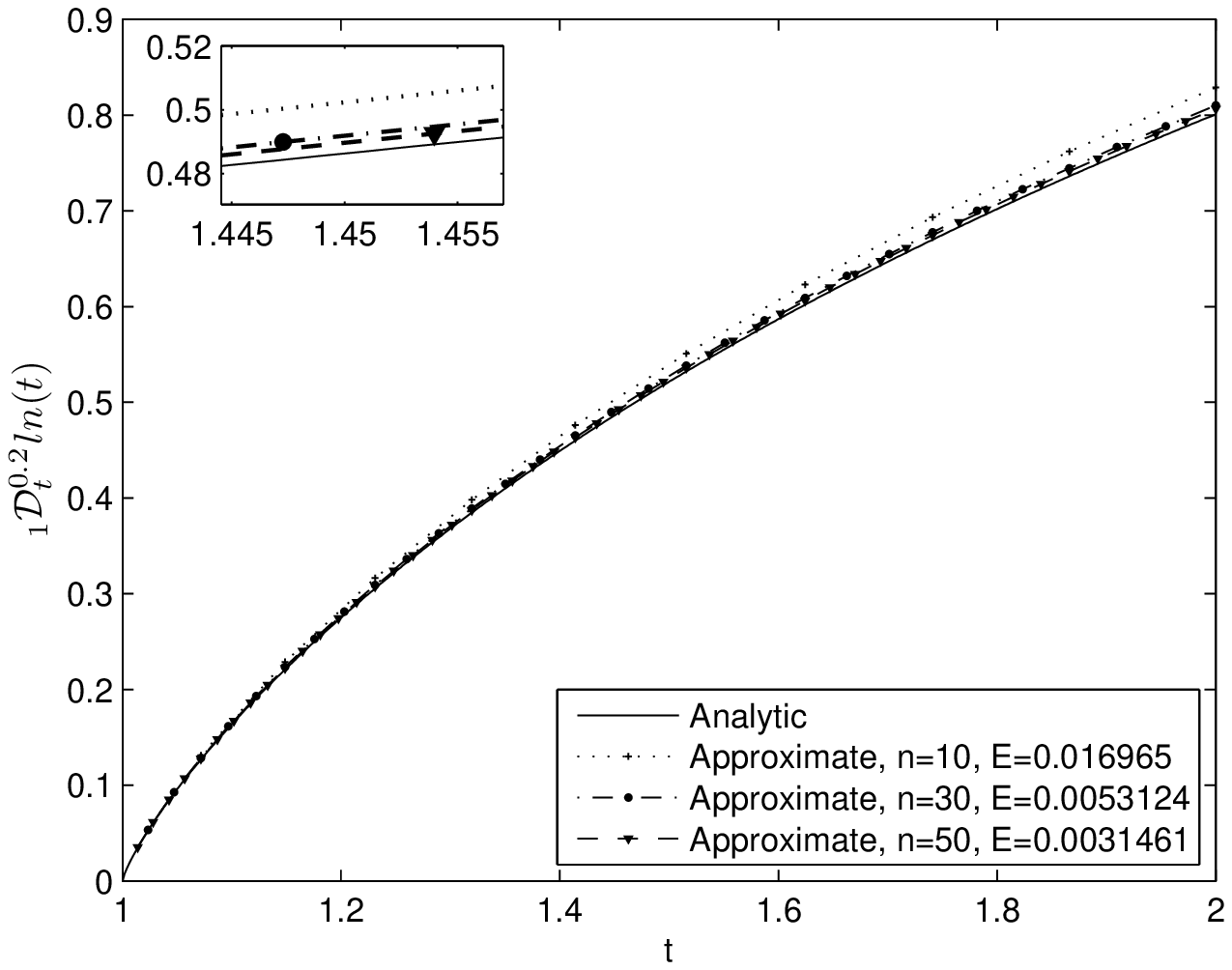}}
    \subfigure[$\a=0.5$]{\label{fig2_c}\includegraphics[scale=0.5]{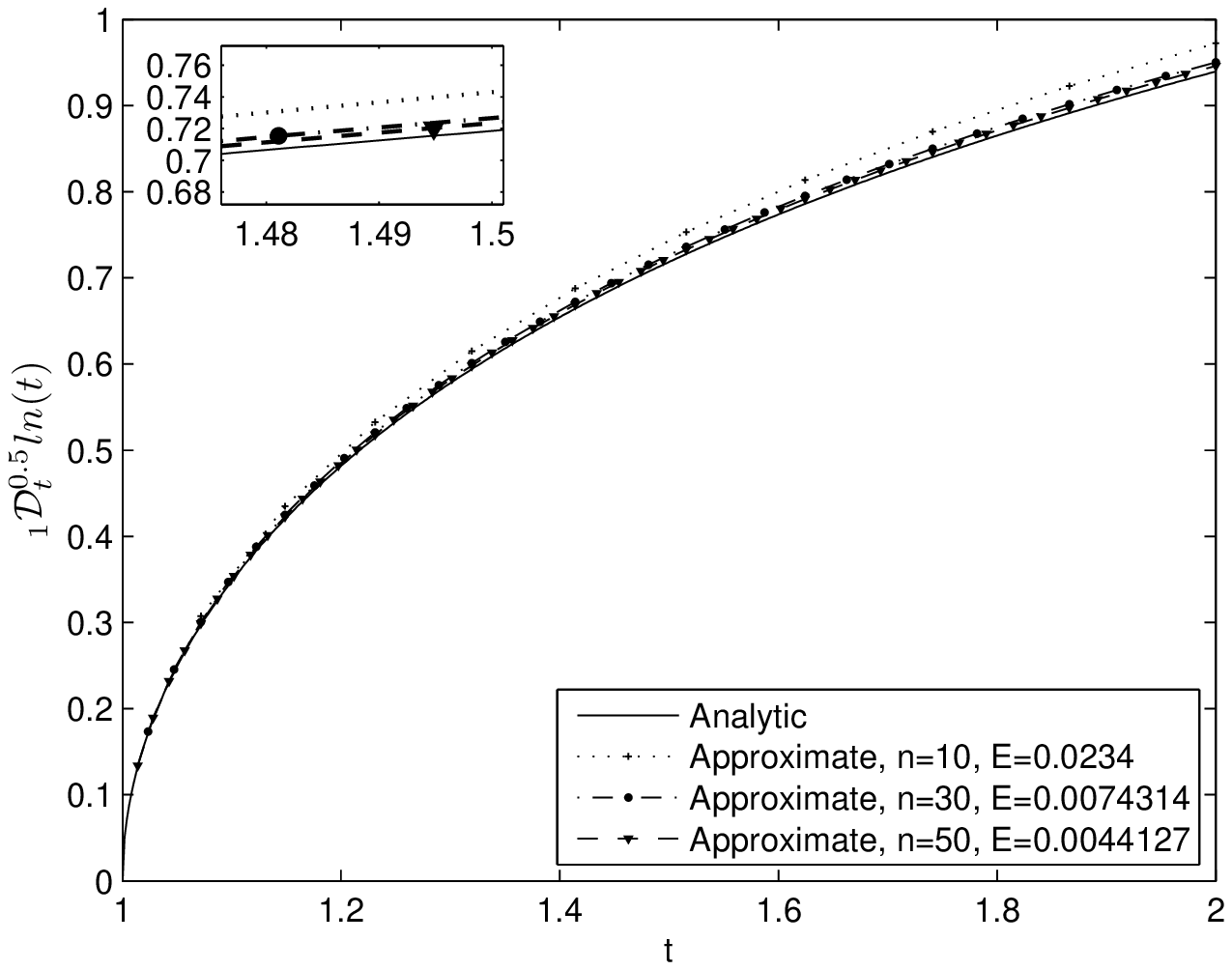}}
    \subfigure[$\a=0.7$]{\label{fig3_c}\includegraphics[scale=0.5]{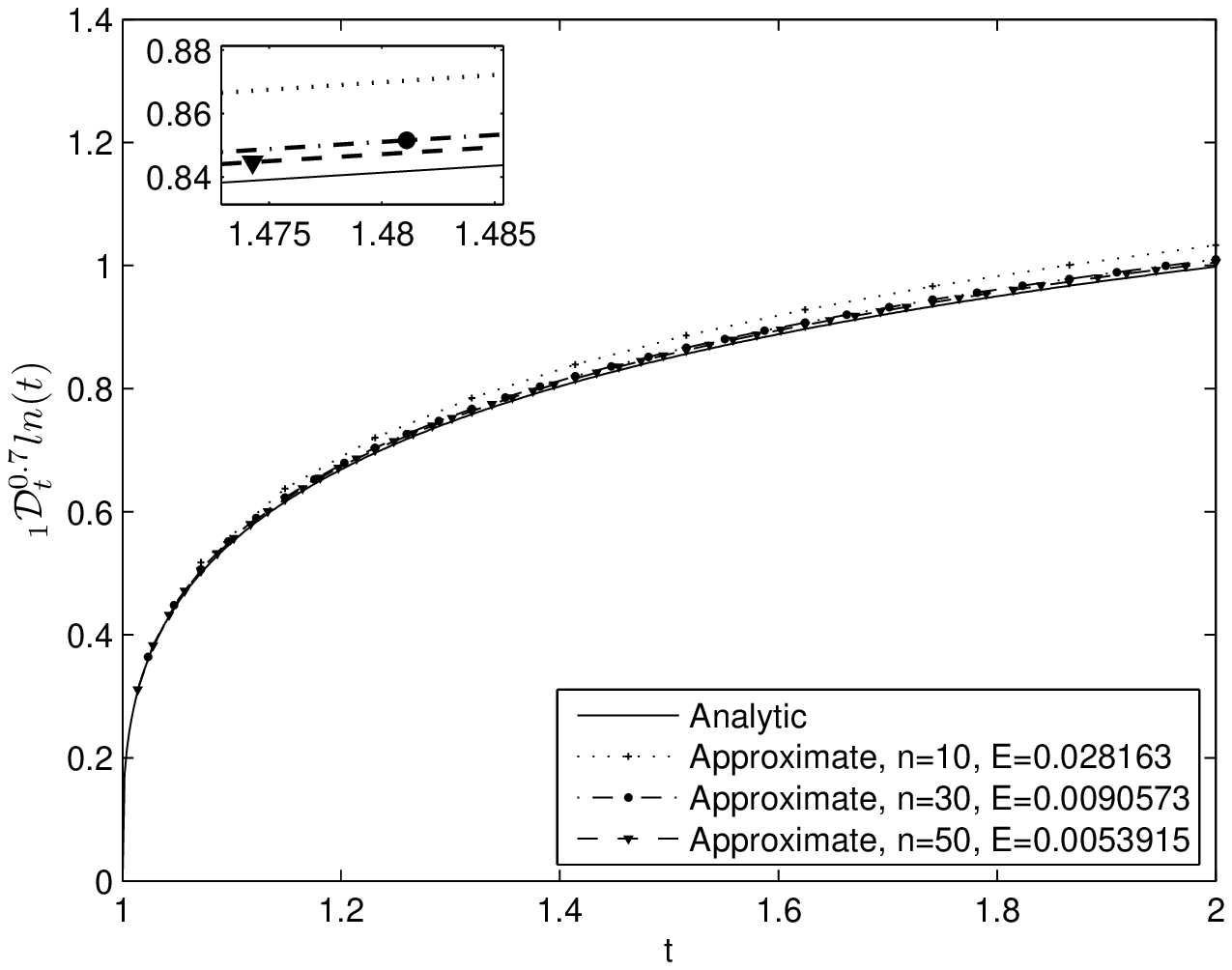}}
    \subfigure[$\a=0.9$]{\label{fig4_c}\includegraphics[scale=0.5]{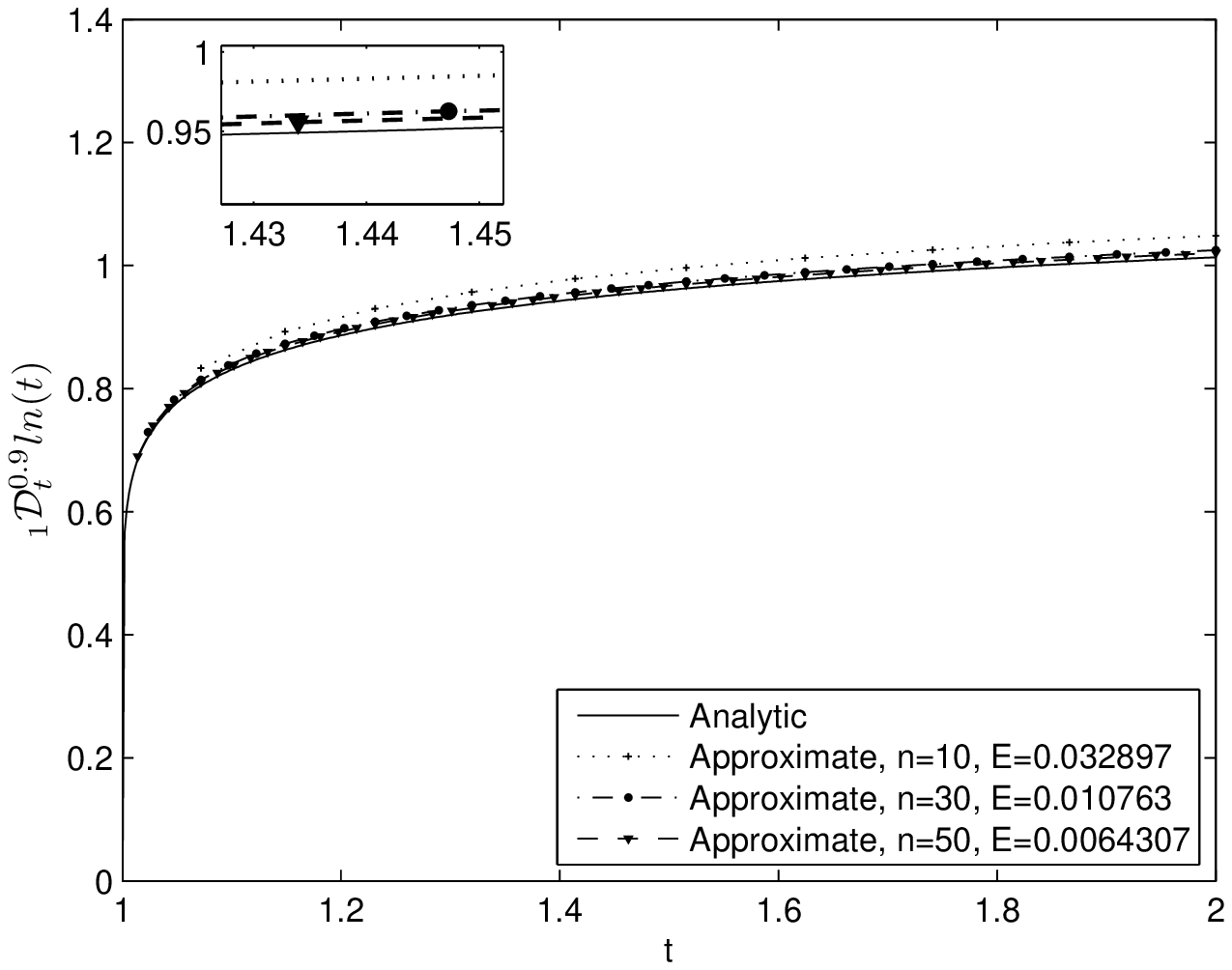}}
  \end{center}
  \caption{Analytic vs. numerical approximation.}
  \label{IntExp}
\end{figure}
We can see that, for a greater value of $n$, the error decreases as expected.

\section{Applications}\label{sec:app}

\begin{example}\label{Ex1} Consider a fractional differential equation with dependence on the left Hadamard fractional derivative:
$$\left\{\begin{array}{l}
\DS f\left(t,x(t),\LHD x(t)\right)=0, \quad t\in[a,b]\\
x(a)=x_a.
\end{array}\right.$$
The procedure to solve numerically the system is described next.
Fix $n\in\mathbb N$ and for $N\in\{1,\ldots,n\}$, define
$$t_N=a \exp(N\t),\quad x_N=x(t_N) \quad \quad \mbox{with} \quad \triangle T=\frac{\ln\frac{b}{a}}{n}.$$
Replacing the fractional operator by the approximation given in Eq. \eqref{appr}, we obtain a classical difference equation with $n$ unknown points
$x_1,\ldots, x_n$,
$$\left\{\begin{array}{l}
\DS f\left(t_N,x_N,\tilde{{_a\mathcal{D}_{t_N}^\a}} x_N\right)=0, \quad N\in\{1,\ldots,n\}\\
x_0=x_a.
\end{array}\right.$$

For example, consider the system
$$\left\{\begin{array}{l}
\DS\HD x(t)+x(t)=\frac{(\ln t)^{1-\a}}{\Gamma(2-\a)}+\ln t, \quad t\in[1,2]\\
x(1)=0.
\end{array}\right.$$
The obvious solution is $\overline x(t)=\ln t$.
Applying the discussed method, we obtain
$$\left\{\begin{array}{l}
\DS\psi\sum_{k=1}^N\left(\omega_{N-k+1}^\a\right)\frac{x_k-x_{k-1}}{\exp(k\t)}\cdot t_k
+x_N=\frac{(\ln t_N)^{1-\a}}{\Gamma(2-\a)}+\ln t_N, \quad N\in\{1,\ldots,n\}\\
x_0=0,
\end{array}\right.$$
with
$$\triangle T=\frac{\ln2}{n}, \quad t_N= \exp(N\t), \quad \psi=\frac{(\t)^{1-\a}}{(1-\exp(-\t))\Gamma(2-\alpha)}
\quad \mbox{and} \quad \omega_k=k^{1-\a}-(k-1)^{1-\a}.$$

In Figure \ref{Example1} we show the numerical results, for different values of $\a\in\{0.2,0.5,0.7,0.9\}$ and for
different values of $n\in\{5,15,30\}$.

\begin{figure}[h!]
  \begin{center}
    \subfigure[$\a=0.2$]{\label{fig_ex_1_1}\includegraphics[scale=0.5]{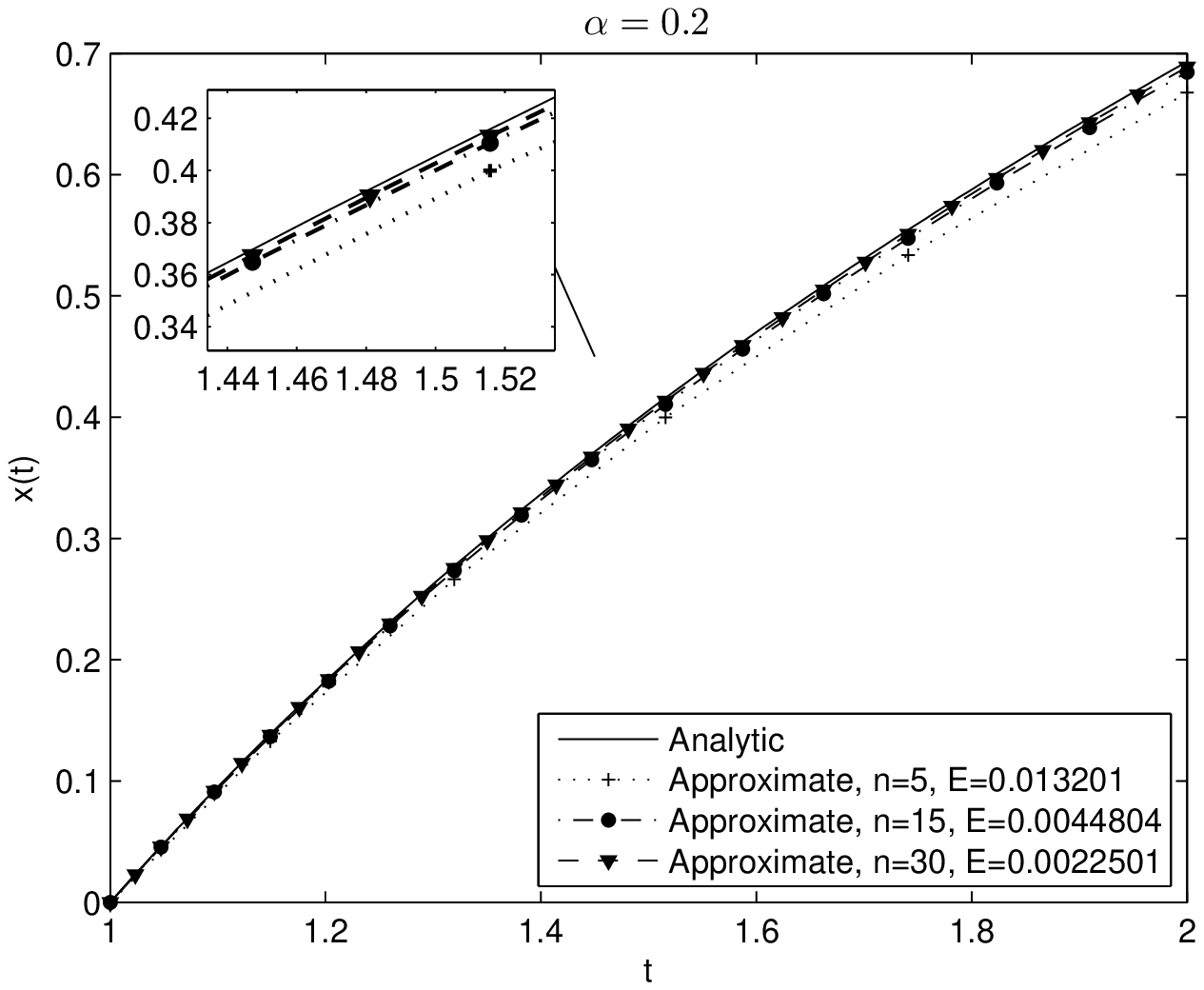}}
    \subfigure[$\a=0.5$]{\label{fig_ex_1_2}\includegraphics[scale=0.5]{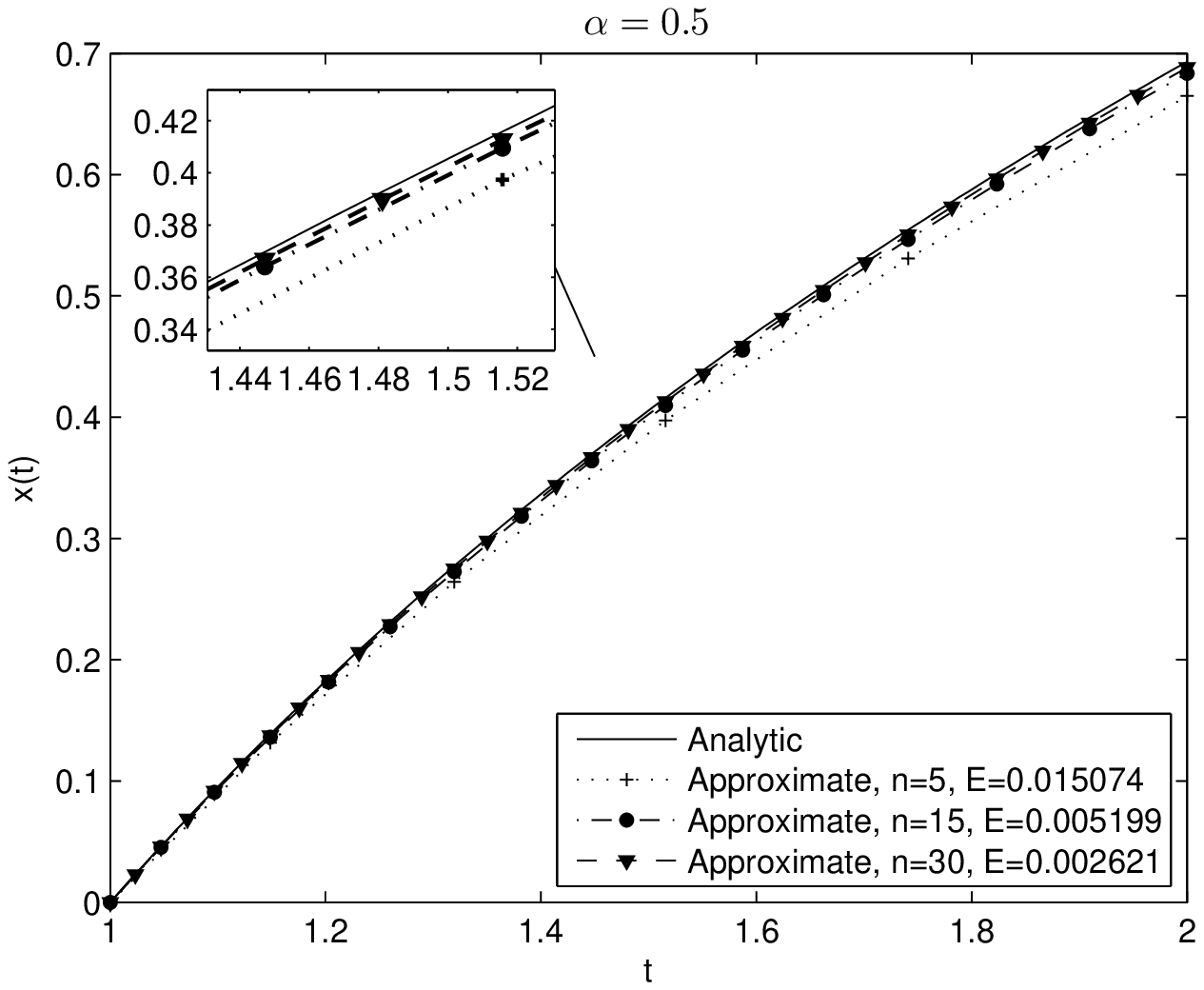}}
    \subfigure[$\a=0.7$]{\label{fig_ex_1_3}\includegraphics[scale=0.5]{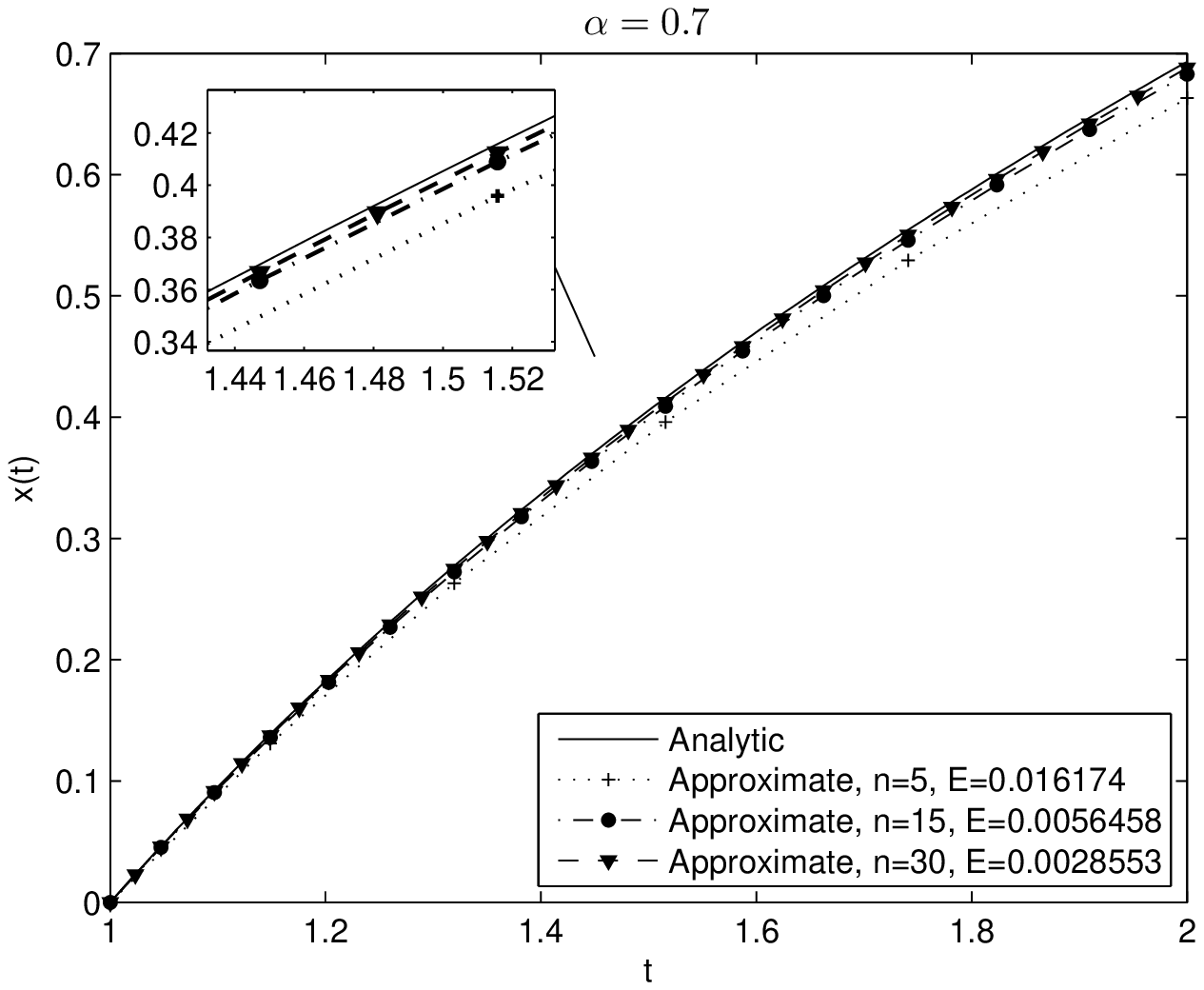}}
    \subfigure[$\a=0.9$]{\label{fig_ex_1_4}\includegraphics[scale=0.5]{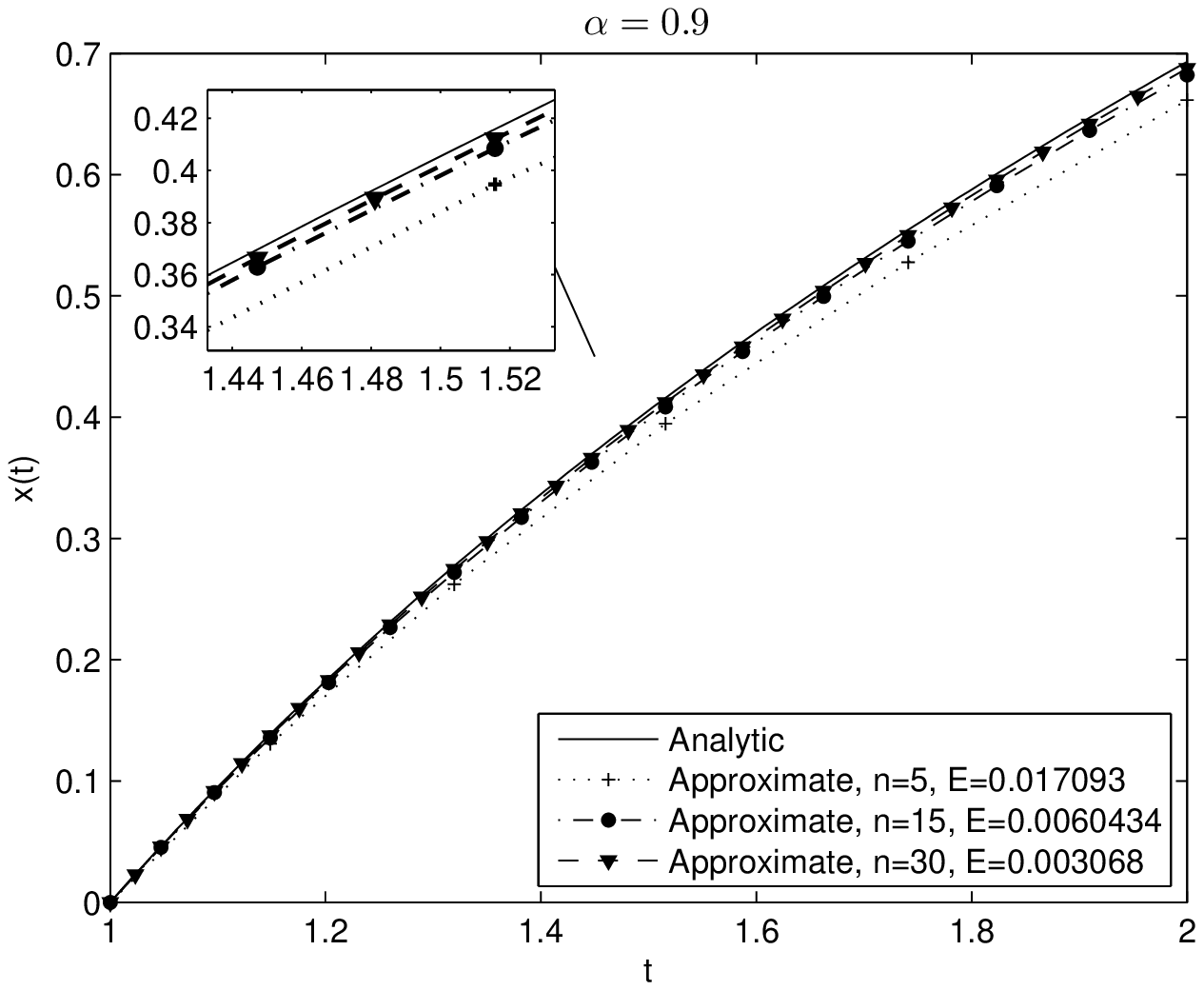}}
  \end{center}
  \caption{Analytic and approximated solutions for problem of Example \ref{Ex1}.}
  \label{Example1}
\end{figure}

\end{example}

\begin{example}\label{Ex2} For our next application, we show how to solve fractional variational problems with the Lagrangian depending on the Hadamard
fractional derivative.
Consider the functional
$$J(x)=\int_a^bL\left(t,x(t),\LHD x(t)\right)\,dt,$$
on the set on functions that satisfy the boundary conditions
$$x(a)=x_a \quad \mbox{and} \quad x(b)=x_b,$$
where $x_a,x_b$ are two fixed reals. The procedure how to find a numerical approximation is explained next.
First, divide the interval $[a,b]$ into $n$ subintervals $[t_{N-1},t_N]$ for
$N\in\{1,\ldots,n\}$, where
$$t_N=a \exp(N\t) \quad \mbox{and} \quad \triangle T=\frac{\ln\frac{b}{a}}{n}.$$
Denoting $x_N=x(t_N)$, applying the trapezoidal rule and taking into consideration Eq. \eqref{appr}, the variational integral is approximated as
$$
\begin{array}{ll}
J(x)&=\DS\sum_{N=1}^n\int_{t_{N-1}}^{t_N}f\left(t,x(t),\LHD x(t)\right)\,dt \\
&\DS\approx\frac{f\left(t_0,x_0,{_a\mathcal{D}_{t_0}^\a} x_0\right)(t_1-t_0)+f\left(t_n,x_n,{_a\mathcal{D}_{t_n}^\a} x_n\right)(t_n-t_{n-1})}{2}\\
&\quad \DS\quad+\sum_{N=1}^{n-1}\frac{f\left(t_N,x_N,{_a\mathcal{D}_{t_N}^\a} x_N\right)}{2}(t_{N+1}-t_{N-1})\\
&\DS\approx\frac{f\left(a,x_a,\tilde{{_a\mathcal{D}_{t_1}^\a}} x_1\right)(t_1-t_0)
+f\left(t_n,x_n,\tilde{{_a\mathcal{D}_{t_n}^\a}} x_n\right)(t_n-t_{n-1})}{2}\\
&\DS\quad+\sum_{N=1}^{n-1}\frac{f\left(t_N,x_N,\tilde{{_a\mathcal{D}_{t_N}^\a}} x_N\right)}{2}(t_{N+1}-t_{N-1}).
\end{array}$$
Observe that we used here the approximation
$${_a\mathcal{D}_{a}^\a} x_a\approx \tilde{{_a\mathcal{D}_{t_1}^\a}} x_1.$$
We can regard this sum as a function of $n-1$ unknown variables $\Psi(x_1,\ldots,x_{n-1})$, and then to find the optimal solution one needs to solve
the system
$$\frac{\partial \Psi}{\partial x_N}=0, \quad \mbox{for} \quad N\in\{1,\ldots, n-1\},$$
and with this we track the desired values  $(x_1,\ldots,x_{n-1})$. Observe that, in opposite to the classical case,
 $\partial \Psi/\partial x_N$ depends on the points $x_N,x_{N+1},\ldots,x_{n-1}$.

For example, we want the global minimizer for
$$J(x)=\int_1^2\left(\HD x(t)-\frac{(\ln t)^{1-\a}}{\Gamma(2-\a)}\right)^2\,dt,$$
with the restrictions
$$x(1)=0 \quad \mbox{and} \quad x(2)=\ln 2.$$
The optimal solution is $\overline x(t)=\ln t$ since the functional takes only non-negative values and vanishes when evaluated at $\overline x$.

In Figure \ref{Example2} we show the solution of the problem, for different values of $\a\in\{0.2,0.5,0.7,0.9\}$ and for
different values of $n\in\{5,15,30\}$.
\begin{figure}[h!]
  \begin{center}
    \subfigure[$\a=0.2$]{\label{fig_ex_2_1}\includegraphics[scale=0.5]{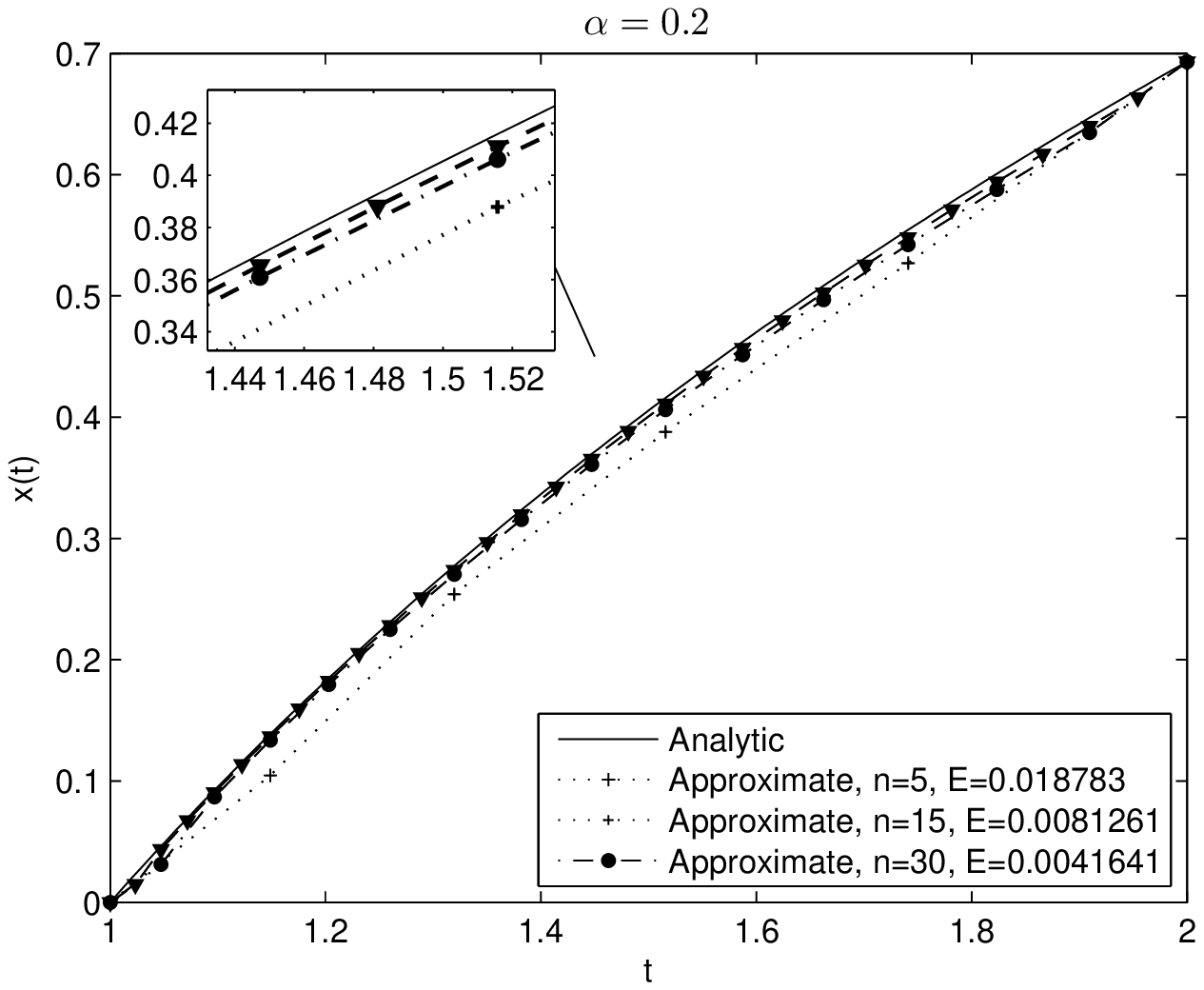}}
    \subfigure[$\a=0.5$]{\label{fig_ex_2_2}\includegraphics[scale=0.5]{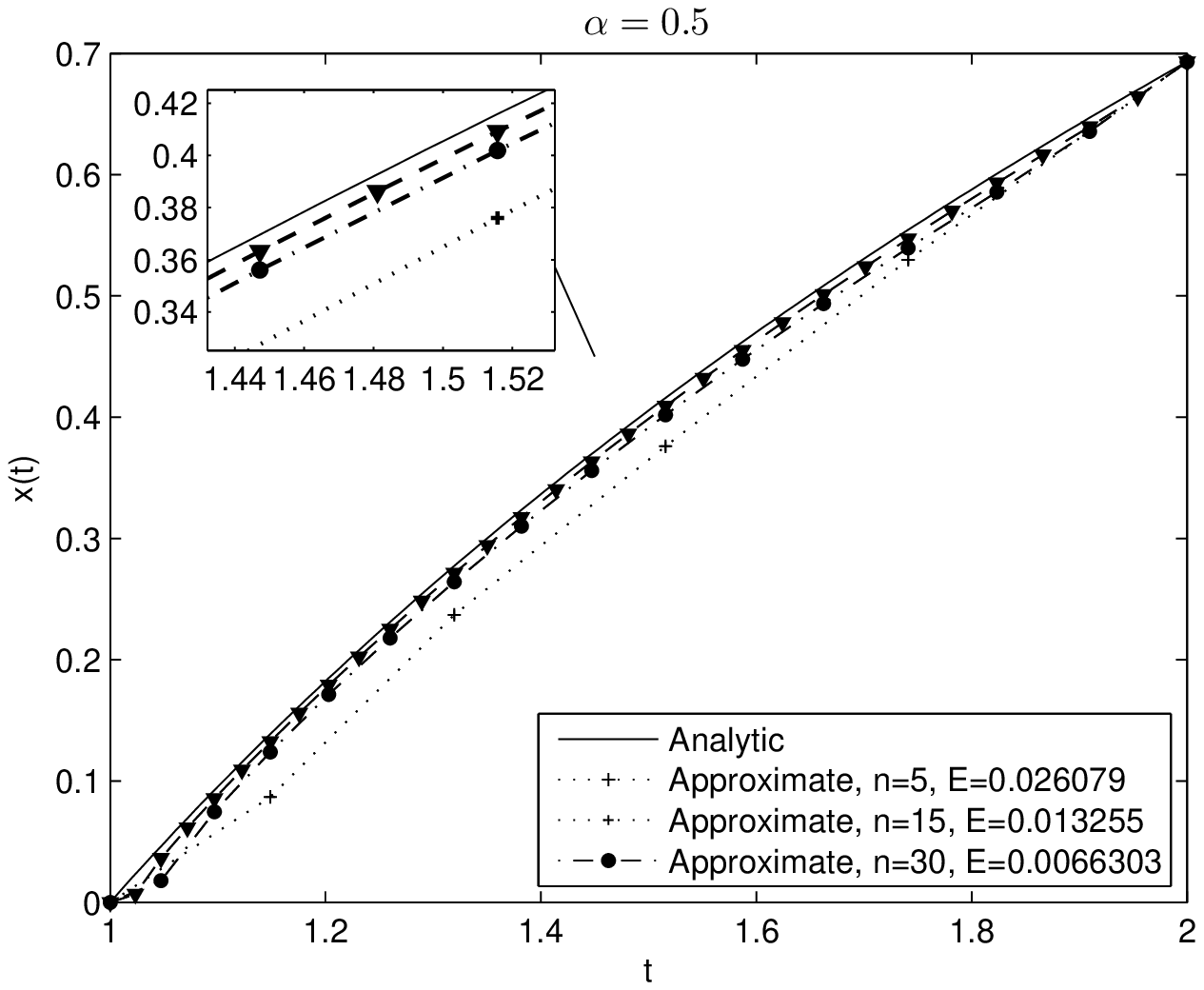}}
    \subfigure[$\a=0.7$]{\label{fig_ex_2_3}\includegraphics[scale=0.5]{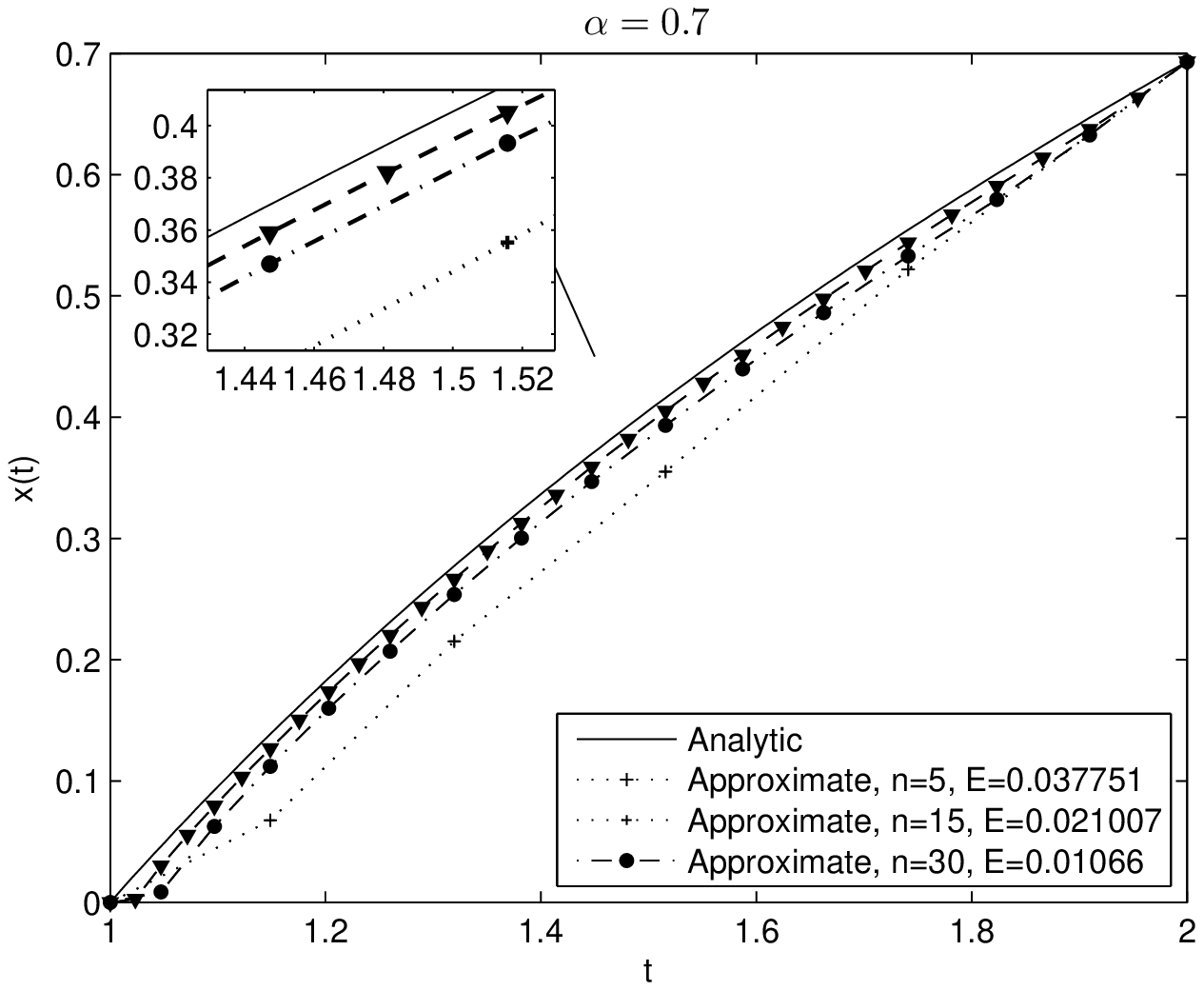}}
    \subfigure[$\a=0.9$]{\label{fig_ex_2_4}\includegraphics[scale=0.5]{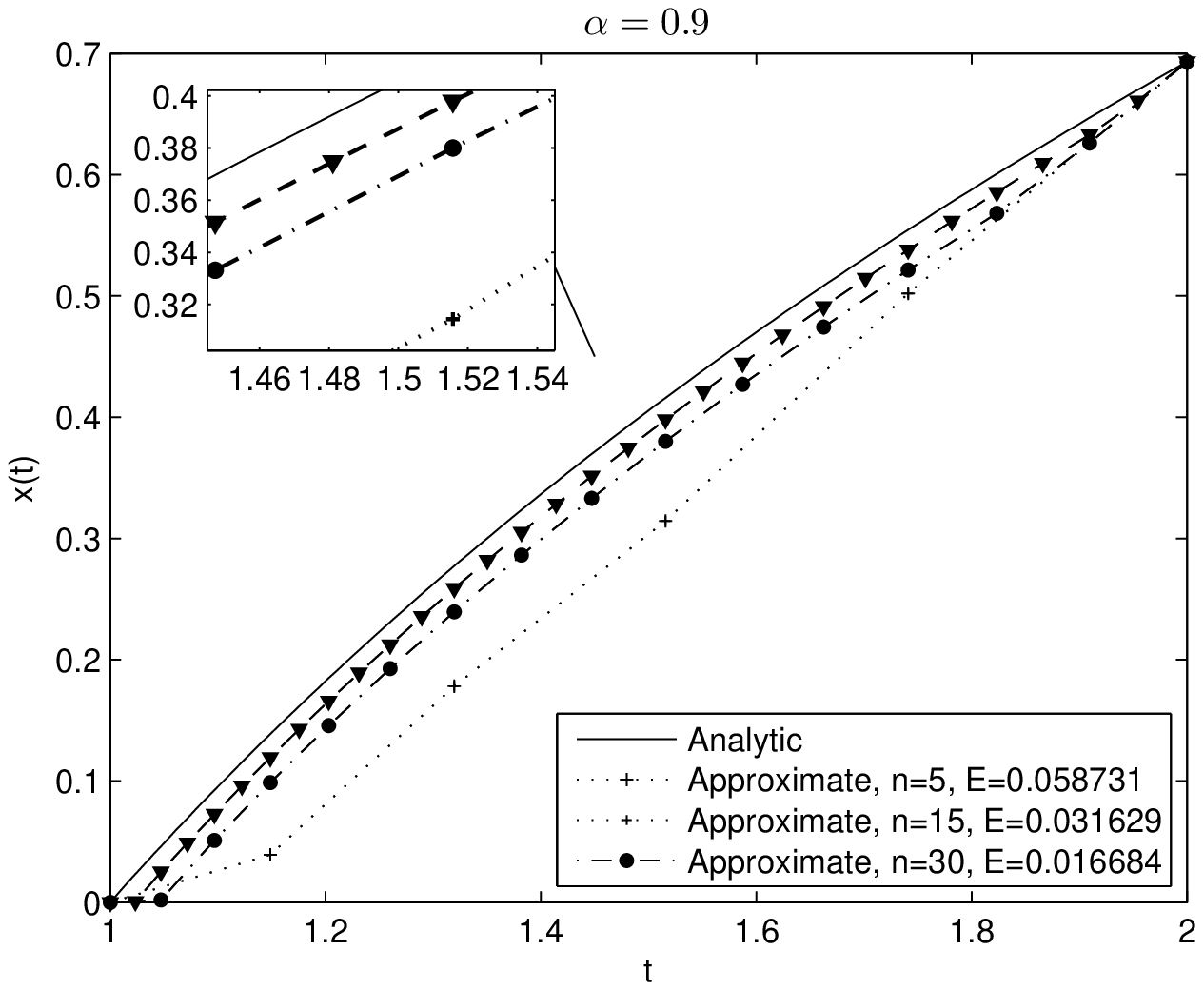}}
  \end{center}
\caption{Analytic and approximated solutions for problem of Example \ref{Ex2}.}
  \label{Example2}
\end{figure}

\end{example}
\FloatBarrier
\section{Conclusions}
For all numerical experiments presented above we used MatLab to obtain the results.
In Examples \ref{Ex1} and \ref{Ex2}, when we take a small number of mesh points ($n=5$) we get a not so good solution.
However, increasing the value of $n$, the error decreases and the numerical solution approaches the analytic solution, converging to it.
From the numerical results we also notice that, for the same values of $n$, as $\alpha$ increases the error also increases,
which makes sense taking into account formula \eqref{Error:Max}. We fix all parameters except $\alpha$, it is easy to check that the maximum
for error value increases as we increase the value of $\alpha$.

\begin{table}[h!]
\center
\begin{tabular}{c|c|cc|cc|cc|}
\cline{2-8}
\multicolumn{1}{c|}{}& $\alpha$ & n & E & n & E & n & E \\
\hline
\multirow{4}{*}{Example \ref{Ex1}} &  0.2   &  5  & $0.013201$   &  15 & $0.0044804$ & 30 & $0.0022501$ \\
                 &  0.5   &  5  & $0.015074$   &  15 & $0.005199$   & 30 & $0.002621$ \\
                 &  0.7   &  5  & $0.016174$   &  15 & $0.0056458$ & 30 & $0.0028553$ \\
                 &  0.9   &  5  & $0.017093$   &  15 & $0.0060434$ & 30 & $0.003068$ \\
\hline
\multirow{4}{*}{Example \ref{Ex2}} &  0.2   &  5  & $0.018783$   &  15 & $0.0081261$ & 30 & $0.0041641$ \\
                 &  0.5   &  5  & $0.026079$   &  15 & $0.013255$ & 30 & $0.0066303$ \\
                 &  0.7   &  5  & $0.037751$   &  15 & $0.021007$ & 30 & $0.01066$ \\
                 &  0.9   &  5  & $0.058731$   &  15 & $0.031629$ & 30 & $0.016684$ \\
\hline
\end{tabular}
\caption{Number of mesh points, $n$, with corresponding  error, $E$ from formula \eqref{Error}.}\label{table}
\end{table}
\FloatBarrier

\section*{Acknowledgments}

This work was supported by Portuguese funds through the CIDMA - Center for Research and Development in Mathematics and Applications,
and the Portuguese Foundation for Science and Technology (FCT-Funda\c{c}\~ao para a Ci\^encia e a Tecnologia), within project UID/MAT/04106/2013.



\label{lastpage}


\end{document}